\documentclass{amsart}
\usepackage[utf8]{inputenc}

\usepackage{amssymb,amsmath, amsthm, amsfonts,comment}
\usepackage{mathrsfs}
\usepackage[ruled,linesnumbered]{algorithm2e}

\usepackage[margin=1in]{geometry}
\usepackage{graphicx, subcaption}
\usepackage{xcolor}
\usepackage{tikz-cd}
\usepackage[doi=false,isbn=false,url=false,eprint=false,style=alphabetic,maxbibnames=99]{biblatex}
\def\rk{\mathop{\rm rk}\nolimits}

\def\spann{\mathop{\rm span}\nolimits}

\def\im{\mathop{\rm im}}
\def\conv{\mathop{\rm conv}}

\def\py{\mathop{\rm py}}

\renewcommand{\int}{\mathrm{int}}
\def\Proj{\mathop{\rm Proj}}

\newcommand{\bbZ}{\mathbb{Z}}

\newcommand{\bbR}{\mathbb{R}}
\newcommand{\bbC}{\mathbb{C}}

\newcommand{\bbP}{\mathbb{P}}
\renewcommand{\P}{\mathbb{P}}

\newcommand{\cV}{\mathcal{V}}

\newcommand{\cN}{\mathcal{N}}

\usepackage{thmtools, thm-restate}
\declaretheorem[numberwithin=section]{theorem}

\newtheorem{lemma}[theorem]{Lemma}
\newtheorem{proposition}[theorem]{Proposition}

\theoremstyle{definition}
\newtheorem{example}[theorem]{Example}

\newtheorem{remark}[theorem]{Remark}

\usepackage{hyperref}
\hypersetup{
    colorlinks=true,
    linkcolor = blue,
    citecolor = blue,
    urlcolor = blue
}

\usepackage{cleveref}

\title{Pythagoras Numbers for Ternary Forms}

\author{Grigoriy Blekherman \and Alex Dunbar \and Rainer Sinn}

\begin{document}

\begin{abstract}
    We study the \emph{Pythagoras numbers} $\py(3,2d)$ of real ternary forms, defined for each degree $2d$ as the minimal number $r$ such that every degree $2d$ ternary form which is a sum of squares can be written as the sum of at most $r$ squares of degree $d$ forms. Scheiderer showed that $d+1\leq \py(3,2d)\leq d+2$. We show that $\py(3,2d) = d+1$ for $2d = 8,10,12$. The main technical tool is Diesel's characterization of height 3 Gorenstein algebras. 
\end{abstract}

\maketitle

\section{Introduction}

A central object in real algebraic geometry is the cone of sums of squares of homogeneous polynomials (forms). Forms usually have multiple representations as sums of squares, and the \emph{length} of a form $f$ is the minimum number $r$ such that $f$ is a sum of $r$ squares, i.e.,~$f = \sum_{i = 1}^r h_i^2$ for some degree $d$ forms $h_1,h_2,\ldots, h_r$. We are concerned with finding the \emph{Pythagoras number} $\py(n,2d)$, defined as the maximum length of any form in $n$ variables of degree $2d$ that is a sum of squares. 
In terms of Gram matrices connecting sums of squares to semidefinite programming \cite{zbMATH06125965, Chua2017GramSpectrahedra}, the Pythagoras number provides information on the existence of low-rank Gram matrices for a form $f$. If a form $f$ in $n$ variables of degree $2d$ is a sum of squares, then there is a positive semidefinite matrix $A$ of rank $\rk(A) \leq \py(n,2d)$ with the property $f(x) = [x]_d^\top A [x]_d$, where $[x]_d$ is a vector of all degree $d$ monomials in the variables $x_1,x_2,\ldots, x_n$. 

Pythagoras numbers are so far only known exactly in very few cases;  bounds on Pythagoras numbers are well-studied in real algebraic geometry starting with Hilbert's seminal paper \cite{Hilbert1888Ueber}. Since we can diagonalize quadratic forms in any number of variables, we have $\py(n,2) = n$ for any $n\geq 1$. Hilbert showed in \cite{Hilbert1888Ueber} that $\py(3,4)=3$. Modern approaches to this theorem in \cite{Blekherman2019LowRankSOSMinimalDegree,Blekherman2022QuadraticPersistence,Chua2017GramSpectrahedra,Scheiderer2017SOSLengthRealForms} have led to new results. However, the exact Pythagoras number has only been determined in few additional cases since Hilbert's work. Scheiderer showed that $\py(3,6)=4$ and $\py(4,4)=5$ \cite{Scheiderer2017SOSLengthRealForms}; both of these cases correspond to varieties of almost minimal degree for which a more general result was shown in \cite{Chua2017GramSpectrahedra}. In our paper, we establish three additional cases of Pythagoras number for ternary forms ($n=3$).

For ternary forms, it is known that $d+1\leq \py(3,2d)\leq d+2$ and that $\py(3,2d) = d+1$ for $2d = 4,6$ \cite{Blekherman2019LowRankSOSMinimalDegree,Chua2017GramSpectrahedra,Scheiderer2017SOSLengthRealForms, Blekherman2022QuadraticPersistence}. Our main result is that the lower bound $d+1$ is also sharp for $2d = 8,10,12$. 

\begin{theorem}\label{thm:IntroMainThm}
    If $2d = 8$, $2d = 10$, or $2d = 12$, then $\py(3,2d) = d+1$. 
\end{theorem}

More generally, we conjecture that $\py(3,2d)=d+1$ for all $d \geq 2$. However, the proof of our main result for $d = 4,5,6$ relies on a case analysis. For $2d = 10$, we have to discuss $3$ cases and for $2d = 12$ we have $7$. For $2d = 14$, the number already increases to $22$. A proof of the conjecture requires therefore a more systematic approach.

\subsection{Outline of the paper}
We sketch the key ingredients of our result. Let us fix the notation
\[\Sigma_{n,2d} = \left\{ f \in \bbR[x_1,x_2,\ldots, x_n]_{2d} \; \middle \vert \; f = \sum_{i = 1}^r h_i^2 \text{ for some } h_1,h_2,\ldots, h_r \in \bbR[x_1,x_2,\ldots, x_n]_d\right\}\]
for the cone of sums of squares so that the Pythagoras number can be formally defined as  

\[\begin{aligned}\py(n,2d)
&= \max \{r \; \vert \; \text{ There exists } f\in \Sigma_{n,2d} \text{ with length } r\}\\
 &= \min \left\{r \; \middle \vert \; \text{For each } f \in \Sigma_{n,2d} \text{ there exists } h_1,h_2,\ldots, h_r\in \bbR[x_1,\ldots, x_n]_d \text{ with } f = \sum_{i = 1}^r h_i^2\right\}.
 \end{aligned}\]

The key observation for our proof of Theorem \ref{thm:IntroMainThm} is a connection between elements $f \in \Sigma_{n,2d}$ which can be written as a sum of strictly fewer than $\py(n,2d)$ squares and Artinian Gorenstein algebras with socle degree $2d$. Recall that an Artinian Gorenstein algebra with socle degree $2d$ is a ring of the form $\bbR[x_1,x_2,\ldots, x_n]/I$, where the ideal $I$ is defined in terms of a linear functional $L$ on $\bbR[x_1,x_2,\ldots, x_n]_{2d}$ as follows:

\[ I = I(L) = \left\{f \in \bbR[x_1,x_2,\ldots, x_n] \; \vert \; \deg(f) > 2d \text{ or } L(fg) = 0 \text{ for all } g \in \bbR[x_1,x_2,\ldots, x_n]_{2d - \deg(f)}\right\}.\]
An ideal of this form is called a \emph{Gorenstein ideal} with socle degree $2d$. To make this connection, we consider the sum-of-$(d+1)$-squares map

\[\Phi_d:\bigoplus_{i = 0}^{d} \bbR[x_1,x_2,x_3]_d \to \bbR[x_1,x_2,x_3]_{2d} \qquad \Phi_d(h_0,h_1,\ldots, h_d) = \sum_{i = 0}^{d}h_i^2.\]
If we had $\py(3,2d) = d+2$, then the image of $\Phi_d$ would be strictly contained in $\Sigma_{3,2d}$. However, the image $\im(\Phi_d) \subseteq \bbR[x_1,x_2,x_3]_{2d}$ has nonempty interior so that its algebraic boundary has codimension $1$. At a general point $p = \sum_{i = 0}^dh_i^2$ on this boundary, the tangent space to $\im(\Phi_d)$ is the degree $2d$ piece of the ideal generated by the $h_i$. That is, $T_p(\im\Phi_d) = \langle h_0,h_1,\ldots h_d \rangle_{2d}$. We show in Section \ref{sec:GorensteinConstruct} that if $p = \sum_{i = 0}^dh_i^2$ is a generic point on the boundary of $\im(\Phi_d)$ which is also an interior point of the cone $\Sigma_{3,2d}$, then $T_p(\Phi_d)$ must have codimension one in $\bbR[x_1,x_2,x_3]_{2d}$ and the forms $h_i$ are contained in a Gorenstein ideal $I(L)$, where $L$ is a nonzero linear functional vanishing on the tangent hyperplane $\langle h_0,h_1,\ldots, h_d \rangle$.  

To derive a contradiction to the existence of such boundary points, we leverage the structure of the associated Gorenstein ideal. Gorenstein ideals in $\bbR[x_1,x_2,x_3]$ are classified by the Buchsbaum-Eisenbud structure theorem, and in \cite{Diesel1996Irreducibility}, Diesel provides a combinatorial characterization of all such ideals. 

In Section \ref{sec:Minimize}, we use a Euclidean distance problem to show that if $p = \sum_{i = 0}^d h_i^2$ is a generic point on the boundary of $\im(\Phi_d)$ which is also an interior point of $\Sigma_{3,2d}$, then there is $q \in \Sigma_{3,2d} \setminus \im(\Phi_d)$, such that $p$ is a local minimizer over $g \in \im(\Phi_d)$ of $\|q-g\|^2$. The optimality conditions for the minimization of $\|q-g\|^2$ above place strong restrictions on the syzygies of $h_0,h_1,\ldots, h_d$. Indeed, the linear functional $L(f) = \langle q-p,f \rangle$ must vanish on $\langle h_0,h_1,\ldots, h_d\rangle$ and $L(\sum_{i =0}^d a_i^2) \leq 0$ for any $a_0,a_1,\ldots, a_d \in \bbR[x_1,x_2,x_3]_d$ with $\sum_{i = 0}^da_ih_i = 0$. Moreover, if $p$ is a minimizer of the distance $\|q-g\|^2$, then $p$ must have length $d+1$. Similar results are proved in \cite{BMinPrep}, and here we combine them with structure results on Gorenstein ideals to derive a contradiction.

Diesel's results \cite{Diesel1996Irreducibility} provide an explicit list of Gorenstein ideals of socle degree $2d$. In Section \ref{sec:GorensteinDimensions}, we derive conditions on the dimensions of the graded pieces of the Gorenstein ideals associated to boundary points of $\im(\Phi_d)$ which lie in the interior of $\Sigma_{3,2d}$. For $2d = 8,10,12$, there are few possibilities for generator and relation degrees of Gorenstein ideals associated to boundary points of $\im(\Phi_d)$: the only possibility for degree $2d = 8$ is a complete intersection of a cubic and two quartics, there are three cases to consider for $2d = 10$, and seven cases for $2d = 12$. 

In Section \ref{sec:MainThmProof}, we consider these cases and derive contradictions with the optimality criteria discussed above. For the cases of $2d = 10$ and $2d = 12$, we show that the structure of the Gorenstein ideal associated to any point $p = \sum_{i = 0}^{d}h_i^2$ on the boundary of $\im(\Phi_d)$ enforces that $p$ has length at most $d$. To make this precise, we define the \emph{Pythagoras number} $\py(X)$ of a real projective variety $X \subseteq \bbP^{r-1}$ with homogeneous coordinate ring $R$ to be the minimum integer $s$ such that if $f \in R_2$ is a sum of squares of elements of $R_1$, then $f$ is a sum of squares of at most $s$ elements of $R_1$. This is consistent with the notion of Pythagoras number of forms as defined above since $\py(n,2d) = \py(\nu_d(\bbP^{n-1}))$, where $\nu_d$ is the $d^{th}$ Veronese embedding. 

In the degree $2d = 10$ and $2d = 12$ cases, we show that if $J$ is a candidate Gorenstein ideal then there is a basis $\{t_1,t_2,\ldots, t_r\}$ of $J_d$ as a real vector space, such the quadratic relations among the $t_i$ are the same as the quadratic relations in the defining ideal of a toric variety $X_K$. Moreover, $\py(X_K)$ is computable and provides an upper bound on the length of a sum of squares of elements of $\bbR[t_1,t_2,\ldots, t_r]_d$. 

Surprisingly, the $2d = 8$ case is more technical. Here, the only candidate Gorenstein ideal $J$ is defined by a complete intersection of a cubic and two quartics. The upper bound provided by the construction of toric variety as in the degree $2d = 10$ and $2d = 12$ cases is $5 = d+1$, which does not provide the desired contradiction. However, by leveraging the second-order optimality conditions and the structure of the complete intersection, we once again obtain a contradiction. 

\subsection{Notation and Conventions}

We fix $S = \bbR[x_1,x_2,x_3]$ and denote by $S_d \subseteq S$ the vector space of degree $d$ forms. Similarly, if $I \subseteq S$ is a homogeneous ideal, then $I_d\subseteq S_d$ is the $d^{th}$ graded piece.  If $h_0,h_1,\ldots, h_r \in S$, the ideal they generate is $\langle h_0,h_1,\ldots, h_r\rangle$. The cone of degree $2d$ forms which are sums of squares is $\Sigma_{2d}:=\Sigma_{3,2d}\subseteq S_{2d}$.

For a real vector space $V$, the dual vector space is $V^*$. If $K \subseteq V$ is a convex cone, then we write $K^* = \{L \in V^* \; \vert \; L(x) \geq 0 \text{ for all } x \in K\}$ for its dual convex cone. If $L \in V^*$ is nonzero, the hyperplane annihilated by $L$ is $L^\perp = \{v \in V \; \vert \; L(v) = 0\}$. For fixed $d \geq 0$, we identify $S_d$ with $S_d^*$ through the \emph{apolar inner product}, where $x_i$ acts by $\frac{\partial}{\partial x_i}$. 

If $W \subseteq \bbR^n$, we denote by $\overline{W}$ the closure of $W$ and $\int(W)$ the interior of $W$. If $W$ is closed, its boundary is $\partial W = W \setminus \int(W)$. 

\section{Pythagoras Number to Gorenstein Ideals}\label{sec:GorensteinConstruct}

In this section we consider the set of sums of squares of length at most $d+1$ and prove Proposition \ref{prop:IntroGorenstein}. Fix $S = \bbR[x_1,x_2,x_3]$. We follow a similar strategy to \cite{Blekherman2019LowRankSOSMinimalDegree,Blekherman2022QuadraticPersistence,Hilbert1888Ueber}. Define the map 

\begin{equation}\label{eq:Phi}
\Phi_d: \bigoplus_{i = 0}^{d} S_d \to S_{2d} \qquad \Phi_d(h_0,h_1,\ldots, h_{d}) = \sum_{i = 0}^{d}h_i^2.
\end{equation}

We observe that $\Phi_d$ induces a map on the projectivizations $\bbP(\bigoplus_{ i =0}^{d} S_d) \to \bbP(S_{2d})$ and therefore is closed and proper. The differential of $\Phi_d$ at the point $h = (h_0,h_1,\ldots, h_{d})$ is the map 

\begin{equation}\label{eq:dPhi}
{\rm d}_h\Phi_d: \bigoplus_{i = 0}^{d} S_d \to S_{2d} \qquad {\rm d}_h\Phi_d(g_0,g_1,\ldots, g_{d}) = 2\sum_{i = 0}^{d}h_ig_i.
\end{equation}

We first record the fact that the image of $\Phi_d$ has nonempty interior. In fact, it is known that that a general form $f\in S_{2d}$ has a representation as a sum of at most $2^2 = 4$ four \emph{complex} squares \cite[Theorem 4]{Froberg2012Waring} and that this implies that the set of sums of four squares of real forms of degree $d$ has nonempty interior in $\Sigma_{2d}$ \cite[Section 2]{Scheiderer2017SOSLengthRealForms}.    

\begin{lemma}[See e.g. \cite{Froberg2012Waring,Scheiderer2017SOSLengthRealForms}]\label{lem:Phi_has_interior}
For $d\geq 3$, $\mathrm{im}(\Phi_d)\subseteq S_{2d}$ has nonempty interior. 
\end{lemma}

It follows from Lemma \ref{lem:Phi_has_interior} that if $\py(3,2d) \geq d+2$, then $\partial\, \overline{\mathrm{int}(\im \Phi_d)} \cap \mathrm{int}(\Sigma_{2d}) \not = \emptyset$. 

We now prove Proposition \ref{prop:IntroGorenstein}, which states that if $p = \sum_{i = 0}^{d} h_i^2$ is a generic element on the boundary of the sums of squares of length at most $d+1$, then the ideal $I = \langle h_0,h_1,\ldots, h_{d}\rangle$ is contained in a Gorenstein ideal with socle degree $2d$ and $I_d$ is base point free.  

\begin{proposition}\label{prop:IntroGorenstein}
If $p = \sum_{i =0}^{d}h_i^2$ is a generic element of $\partial \,\overline{\mathrm{int}(\im \Phi )} \cap \mathrm{int}(\Sigma_{2d})$ and $I = \langle h_0,h_1,\ldots, h_d \rangle$ is the ideal generated by the forms $h_0,h_1,\ldots, h_d$, then $I_d$ is base point free and $I$ is contained in a Gorenstein ideal with socle degree $2d$. 
\end{proposition}

\begin{proof}

If $\py(3,2d) = d+2$, then $\mathrm{im}(\Phi_d)$ is a proper subset of $\Sigma_{2d}$ which has interior in $S_{2d}$ by Lemma \ref{lem:Phi_has_interior}. Set $B = \partial \, \overline{\mathrm{int}(\mathrm{im}\Phi_d)}$ to be the boundary of the closure of the interior of the image of the map $\Phi_d$.  It follows from \cite{Sinn15AlgebraicBoundaries} that $\operatorname{codim}(B) = 1$ since $\int(\im \Phi_d)$ and $S_{2d} \setminus \int(\im \Phi_d)$ are both regular subsets of $S_{2d}$ equipped with the Euclidean topology, i.e., they are each contained in the closure of their interiors. 

By Sard's Theorem (see e.g; \cite{Sard1942MeasureCriticalValues},\cite[Theorem 9.6.2]{BCR1998RealAlgebraicGeometry}), at a general point $p = \sum_{i =0}^{d}h_i^2 \in B$, we have that $T_pB$ has dimension $\dim_{\bbR}S_{2d} - 1 = \binom{2 + 2d}{2d} - 1$ and that $\rk {\rm d}_h \Phi = \dim T_p B$. By the calculation \eqref{eq:dPhi}, this implies that $\langle h_0,h_1,\ldots, h_{d} \rangle_{2d} \subseteq S_{2d}$ is a hyperplane. Let $L \in S_{2d}^*$ be a nonzero linear functional annihilating the hyperplane $\langle h_0,h_1,\ldots, h_{d}\rangle$. Then $\langle h_0,h_1,\ldots, h_{d}\rangle$ is contained in the Gorenstein ideal $I(L)$, which has socle in degree $2d$. 

We now show that for general $p \in B \cap \mathrm{int}(\Sigma_{2d})$, we can take $\langle h_0,h_1,\ldots, h_{d} \rangle_d$ to be a base point free linear series. Suppose for the sake of contradiction that there is $z \in \cV_{\bbC}(h_1,h_2,\ldots, h_d)$.  If $z$ is a real base point of the $h_i$ then $p$ is not in the interior of $\Sigma_{2d}$. If $z$ is a complex base point, then so is its complex conjugate $\bar{z}$ since the $h_i$ are defined over $\bbR$. The set of polynomials which vanish at both $z$ and $\bar{z}$ is codimension $2$ so that a general $p\in B \cap \mathrm{int}(\Sigma_{2d})$ will not vanish at $z$ and $\bar{z}$.  
\end{proof}

\begin{remark}
While the results of this section are stated for ternary forms, the analogous statements hold for $n \geq 4$. In particular, for the case $(n,2d) = (5,4)$, we obtain the upper bound

\[\py(5,4) \leq \max \{\dim(J_2) \; \vert \; J \subseteq \bbR[x_1,x_2,\ldots, x_5] \text{ is Gorenstein with socle degree } 4\} + 1 \]
By \cite[Theorem 3.1]{Migliore2013GorensteinQuadrics}, this implies that $\py(5,4) \leq 8$. On the other hand, by considering the quadratic persistence \cite{Blekherman2022QuadraticPersistence} of the veronese embedding $\nu_2(\bbP^4)$, we see that $\py(5,4) \geq 7$. 
\end{remark}

\section{Minimizing the Distance}\label{sec:Minimize}

In this section, we show that a general element $p = \sum_{i = 0}^{d}h_i^2$ of the boundary of $\overline{\mathrm{int}(\im \Phi_d)}$ is the minimizer of some distance. More precisely, if $\py(3,2d)\geq d+2$, then there is $q \in \mathrm{int}(\Sigma_{2d})$ such that $p$ is the local minimizer over forms $g \in \overline{\mathrm{int}(\im(\Phi_d))}$ of the function $\|q-g\|^2$. We then relate this property to the syzygies of the forms $h_0,h_1,\ldots, h_{d}$. Throughout this section, we denote $M := \partial \, \overline{\mathrm{int}(\im(\Phi_d))} \bigcap \mathrm{int}(\Sigma_{2d})\subseteq S_{2d}$.

\begin{proposition}\label{prop:IntroArgmin}
If $p$ is a generic element of $M$, there exists an element $q \in  \mathrm{int}(\Sigma_{3,2d})$ and a neighborhood $U\subseteq M$ of $p$ such that $q$ cannot be written as a sum of fewer than $d+2$ squares and such that 

\begin{equation}\label{eq:distance_to_image}p = \arg \min_{g \in U} \|q - g\|^2. 
\end{equation}
\end{proposition}

\begin{proof}
Let $p = \sum_{i = 0}^{d}h_i^2$ be a generic element of $M$ so that $H:= \langle h_0,h_1,\ldots, h_{d} \rangle_{2d}$ has codimension $1$ in $S_{2d}$. Let $L \in S_{2d}^*$ be the linear functional which vanishes on $H$, and $f \in S_{2d}$ be such that $L(g) = \langle f, g \rangle$ for all $g \in S_{2d}$. Since $p \in M$, either $p+\epsilon f$ or $p-\epsilon f$ is not an element of $\im(\Phi_d)$ for small $\epsilon > 0$. So, replacing $f$ with $-f$ if necessary, it follows that $q = p + \epsilon f$ cannot be written as a sum of fewer than $d+2$ squares and $q \in \mathrm{int}(\Sigma_{2d})$ for $\epsilon > 0$ small enough . Moreover, since $L$ vanishes on $H$, there is a neighborhood $p \in U \subseteq M$ such that $ p = \arg \min_{g \in U}\|q-g\|^2$. 
\end{proof}

Our primary application of Proposition \ref{prop:IntroArgmin} is to understand the syzygies of the forms $h_0,h_1,\ldots, h_{d}$ when $p = \sum_{i = 0}^{d}h_i^2$ is a generic element of the boundary of $M$. To do so, we examine the first and second-order optimality conditions in the problem \eqref{eq:distance_to_image}. We first need the following technical lemma regarding second-order optimality conditions.

\begin{lemma}\label{lem:second_order_opt}
Let $Q$ and $R$ be symmetric $n\times n$ matrices with $Q \succeq 0$. Suppose that $Q + \epsilon R \succeq 0$ for $\epsilon \geq 0$ sufficiently small. Then, the restriction of $R$ to $\ker(Q)$ is positive semidefinite and if $v^\top R v = v^\top Q v = 0$, then $Rv = 0$.  
\end{lemma}

\begin{proof}
Suppose that $Q + \epsilon R \succeq 0$ for $\epsilon \geq 0$ sufficiently small. Let $v \in \ker Q$. Then, $0 \leq v^\top(Q + \epsilon R)v = \epsilon v^\top R v$. So, $R \succeq 0$ on $\ker Q$. It then follows that if $v^\top Rv = v^\top Q v = 0$, then $v \in \ker Q$ since $Q \succeq 0$ and therefore $Rv = 0$ since $R \succeq 0$ on $\ker Q$.  
\end{proof}

We can now prove the main proposition of this section which relates optimality conditions for the problem \eqref{eq:distance_to_image} to syzygies of the forms $h_0,h_1,\ldots, h_{d}$. 

\begin{proposition}\label{prop:Nonneg_SOS_Syzygies}
Let $p = \sum_{i = 0}^{d}h_i^2$ be a general element of $M$ and let $q \in S_{2d}$ be as in Proposition \ref{prop:IntroArgmin}. Set $L \in S_{2d}^*$ to be the linear functional $L(g) = \langle q-p, g\rangle$. Then,

\begin{enumerate}
\item $L(g) = 0$ for all $g \in \langle h_0,h_1,\ldots, h_{d}\rangle_{2d}$.
\item If $a_0,a_1,\ldots, a_{d} \in S_d$ are such that $\sum_{i = 0}^{d}a_ih_i = 0$, then $L\left(\sum_{i = 0}^{d} a_i^2\right) \leq 0$. If, in addition, $L\left(\sum_{i = 0}^{d} a_i^2\right) = 0$, then $a_i \in I(L)_d$ for each $i = 0,1,\ldots, d$.  
\item $\pm L \not \in \Sigma_{2d}^*$
\item The length of $p$ is $d+1$.
\end{enumerate} 
\end{proposition}

\begin{proof}
We examine perturbations of $p$. For $\epsilon \in \bbR$, let $p_\epsilon = \sum_{i = 0}^{d}\left(h_i + \epsilon a_i + \epsilon^2 b_i\right)^2$, where $a_i,b_i \in S_d$ for $i = 0,1,\ldots, d$.  Then,

\[\begin{aligned}
\langle q - p_\epsilon, q-p_{\epsilon} \rangle &= \langle q - p, q- p \rangle - 4\epsilon\left\langle q-p, \sum_{i = 0}^{d}h_ia_i\right\rangle - 2\epsilon^2 \left\langle q-p, \sum_{i = 0}^{d}a_i^2 + 2h_ib_i\right\rangle + \epsilon^2 \left \langle \sum_{i = 0}^{d}h_ia_i,\sum_{i = 0}^{d} h_ia_i \right\rangle
\\
&= L(q-p)- 4\epsilon L\left(\sum_{i = 0}^{d}h_ia_i\right) - 2\epsilon^2 L\left(\sum_{i = 0}^{d}a_i^2 + 2h_ib_i\right) + \epsilon^2 \left \langle \sum_{i = 0}^{d}h_ia_i,\sum_{i = 0}^{d} h_ia_i \right\rangle
\end{aligned}\]

Since $p$ is a local minimum of the distance to $q$, it follows from first order optimality that $L(\sum_{i = 0}^{d}h_ia_i)= 0$ for any choice of $a_0,a_1,\ldots, a_{d} \in S_d$, proving the first claim.  If we further take $a_0,a_1,\ldots, a_{d} \in S_d$ such that $\sum_{i = 0}^{d}h_ia_i = 0$, then 

\[\langle q-p_\epsilon, q-p_\epsilon \rangle = L(q-p) - 2\epsilon^2 L\left(\sum_{i = 0}^{d }a_i^2\right).\]

It then follows that $L(\sum_{i =0}^{d}a_i^2) \leq 0$. By applying Lemma \ref{lem:second_order_opt} to the quadratic forms $Q(a_0,\ldots, a_{d}) = \left \langle \sum_{i = 0}^{d} h_ia_i, \sum_{i = 0}^{d}h_ia_i\right \rangle$ and $R(a_0,\ldots, a_{d}) = -2L(\sum_{i = 0}^{d}a_i^2)$ on $S_d^{d+1}$, it follows that if $L(\sum_{i = 0}^{d}a_i^2) = 0$, then $-L(\sum_{i = 0}^{d}a_ig_i) =0$ for any $(g_0,g_1,\ldots, g_{d}) \in S_d^{d+1}$. So, each $a_i \in I(L)_d$. 

For part $(3)$, note that since $p \in \langle h_0,h_1,\ldots, h_d\rangle$, it follows from part $(1)$ that $L(p) = -L(p) = 0$. Since $p \in \int(\Sigma_{2d})$, the hyperplane $L^\perp$ cannot be a supporting hyperplane of $\Sigma_{2d}$. Therefore $\pm L \not \in \Sigma_{2d}^*$. 

For part $(4)$, suppose for the sake of a contradiction that $p$ has length at most $d$ so that $p = \sum_{i = 0}^{d-1}g_i^2$ for some $g_i \in S_d$. Then, since $\pm L \not \in \Sigma_{2d}^*$, there exists $a \in S_{d}$ such that $L(a^2) > 0$. It then follows that for each $\epsilon > 0$, 

\[\langle q-p,q-(p+\epsilon a^2)\rangle = \|q-p\|^2 - \epsilon L(a^2) < \|q-p\|^2,\]
a contradiction with the construction of $q$. 
\end{proof}

Proposition \ref{prop:Nonneg_SOS_Syzygies} suggests methods for eliminating possible Gorenstein ideals from consideration. In particular, given a potential Gorenstein ideal $I = I(L)$, we will show that if $p = \sum_{i = 0}^{d}h_i^2$ for some $h_0,h_1,\ldots, h_{d} \in I_d$ such that $\langle h_0,h_1,\ldots, h_{d}\rangle_{2d} = I_{2d}$, then either $p$ has length $d$ or that if the stationarity conditions (1) and (2) hold then $L \in \Sigma_{2d}^*$.

In the remainder of this section, we further investigate the syzygies of the forms $h_0,h_1,\ldots, h_{d}$. First, we show that the $h_i$ must be linearly independent. 

\begin{lemma}\label{lem:L_nonneg_q_LI}
Fix $h_0,h_1,\ldots, h_r \in S_d$. Suppose that $L \in S_{2d}^* \setminus \Sigma_{2d}^*$ is a linear functional such that $L(\sum_{i =0}^r a_i^2) \geq 0$ for any $a_0,a_1,\ldots, a_r\in S_{d}$ which satisfy $\sum_{i = 0}^{r}h_ia_i = 0$. Then, $h_0,h_1,\ldots, h_r$ are linearly independent. 
\end{lemma}

\begin{proof}
Suppose for the sake of a contradiction that $a_0,a_1,\ldots, a_r \in \bbR$ are not all zero and $\sum_{i = 0}^r a_ih_i = 0$ and $\sum_{i = 0}^r a_i^2 = 1$. Let $f \in S_{d}$ be arbitrary. Then, $\sum_{i = 0}^r (a_if) h_i = 0$ and therefore $L(\sum_{i = 0}^r(a_if)^2) = L(f^2) \geq 0$, which implies that $L \in \Sigma_{2d}^*$.
\end{proof}

Next, we show that the presence of a nontrivial syzygy of the form $\sum_{0\leq i\leq j \leq r}\alpha_{i,j}h_ih_j = 0$ with each $\alpha_{i,j} \in \bbR$ certifies that the length of $\sum_{i =0}^r h_i^2$ is at most $r$. 

\begin{lemma}\label{lem:syzygy_in_ideal}
    Suppose that $h_0,h_1,\ldots, h_r \in S_d$ and that there are $\alpha_{i,j} \in \bbR$ for $0 \leq i \leq j\leq r$, not all zero, such that $\sum_{0\leq i\leq j \leq r}\alpha_{i,j}h_ih_j = 0$. Then, the length of $\sum_{i = 0}^rh_i^2$ is at most $r$. 
\end{lemma}

\begin{proof}
Let $A$ be the symmetric matrix representative of the quadratic form $\sum_{0\leq i\leq j\leq r} \alpha_{i,j}h_ih_j$. Then,

\[\sum_{i = 0}^{r}h_i^2 = \begin{bmatrix} h_0 & h_1 & \dots & h_r\end{bmatrix}\left(I - tA\right)\begin{bmatrix} h_0 & h_1 & \dots & h_r\end{bmatrix}^\top.\]
Since $A$ is symmetric and not the zero matrix, there is $t \in \bbR$ such that $(I-tA)$ has rank at most $r$. 
\end{proof}

\section{Height Three Gorenstein Algebras}\label{sec:GorensteinDimensions}

In this section, we investigate the possible structures of Gorenstein ideals $J\subseteq S$ with socle degree $2d$ such that $\langle J_d \rangle_{2d} \subseteq S_{2d}$ is a hyperplane, and which contain a base point free linear series in degree $d$. As a starting point, we recall results on the structure of height 3 Gorenstein ideals. A first construction of Gorenstein ideals comes from complete intersections.

\begin{theorem}[{\cite[Theorem CB8]{Eisenbud1996CBTheorems}}] \label{thm:CB8}
If $q_0,q_1,q_2 \in S$ with $\deg(q_i) = d_i$ have no common zeros (real or complex), then the complete intersection $\langle q_0,q_1,q_2 \rangle$ is a Gorenstein ideal with socle in degree $d_0 + d_1 + d_2 - 3$. 
\end{theorem}

Theorem \ref{thm:CB8} imposes the following conditions on low degree generators for Gorenstein ideals $J$ constructed as in Proposition \ref{prop:IntroGorenstein}. First, we obtain a lower bound on the minimal generator degree of $J$.

\begin{proposition}\label{prop:No_degree_two_gens}
Fix $d \geq 3$. If $J \subseteq S$ is a Gorenstein ideal with socle in degree $2d$ which contains a base point free linear series $W$ in degree $d$, then $J$ has no generators of degree 2. 
\end{proposition}

\begin{proof}
Suppose for the sake of a contradiction that $J$ has a generator $f$ of degree 2. Since $W \subseteq J_d$, there are $g,h\in J_d$ such that $\langle f,g,h \rangle_d \subseteq W$ is a base point free linear series. Therefore, by Theorem \ref{thm:CB8}, it must be the case that $\langle f,g,h \rangle$ is Gorenstein with socle in degree $2+d+d -3 = 2d-1$ and therefore $S_{2d}= \langle f,g,h \rangle_{2d} \subseteq J_{2d}$ which contradicts the assumption of the socle degree of $J$. 
\end{proof}

\begin{proposition}\label{prop:LowDegreeGens}
Let $J$ be a Gorenstein ideal with socle degree $2d$. Then

\begin{enumerate}
\item If $J$ has at most two generators of degree at most $d$, then $J_d$ does not contain a basepoint free linear series.
\item If $J$ has exactly three generators $f_1,f_2,f_3$ of degree at most $d$, $\cV_{\bbC}(f_1,f_2,f_3) = \emptyset$, and $\sum_{i = 1}^3 \deg(f_i) > 2d +3$, then $\langle J_d \rangle_{2d}$ is not a hyperplane in $S_{2d}$.   
\end{enumerate}

\end{proposition}

\begin{proof}
For statement $(1)$, note that if $J$ has at most two generators of degree at most $d$, say $f,g \in S$, then $J_d = \langle f, g \rangle_{d}$. This implies that $\cV(f,g)$ is the base locus of $J_d$ and this is at least a $0$-dimensional subset of $\P^2$. In particular, $J_d$ cannot contain a basepoint free linear series. 

For statement $(2)$, since $\cV_{\bbC}(f_1,f_2,f_3) = \emptyset$, the ideal $\langle f_1,f_2,f_3\rangle$ is Gorenstein with socle degree $\sum_{i = 1}^3\deg(f_i) -3 > 2d$. So $\langle f_1,f_2,f_3\rangle_{2d}\subseteq S_{2d}$ has codimension at least 2.\end{proof} 

If $J \subseteq S$ is a Gorenstein ideal with socle degree $2d$, we define the \emph{Hilbert function} $T_J$ to be the Hilbert function of the Artinian algebra $S/J$, i.e. the sequence of $\bbR$-vector space dimensions of the graded pieces of $S/J$:

\[T_J = (\dim (S/J)_0, \dim (S/J)_1, \ldots, \dim (S/J)_{2d-1},\dim (S/J)_{2d}).\]
with $\dim(S/J)_0 = 1$ and $\dim(S/J)_{2d} = 1$. This sequence is always symmetric, that is $\dim(S/J)_{i} = \dim(S/J)_{2d-i}$ for all $0\leq i\leq 2d$. Diesel \cite{Diesel1996Irreducibility} characterizes the possible Hilbert functions $T_J$. 

In order to describe Diesel's result, we need the terminology of \emph{self-complementary partitions}. A partition of a $2d \times (2d-2k+2)$ block is represented graphically by a box with $2d$ rows and $2d -2k + 2$ columns, where row $i$ has $a_i$ filled and $b_i$ unfilled blocks (so that $b_i = 2d-2k+2 - a_i$). A partition is self-complementary if $a_1\leq a_2\leq \ldots \leq a_{2d}$ and if $a_i = b_{2d - i + 1}$ for all $1\leq i \leq 2d$ (see Example \ref{ex:ExamplePartition} and Figure \ref{fig:ExamplePartition} below for an example).

\begin{proposition}[{\cite[Proposition 3.9]{Diesel1996Irreducibility}}]\label{prop:GorensteinPartitions}
    For fixed socle degree $n$ and minimal generator degree $k$, there is a one-to-one correspondence between Hilbert functions of Gorenstein algebras $S/J$ and self-complementary partitions of $2k$ by $n-2k+2$ blocks.
\end{proposition}

\begin{remark}\label{rem:GorensteinBlocks}

Diesel explicitly relates the self-complementary partitions to a free resolution of $S/J$ of the form

\begin{equation}\label{eq:Ternary_Free_Res}0 \to S(-2d-3) \to \bigoplus_{i = 1}^{2k+1}S(-p_i) \to \bigoplus_{i = 1}^{2k+1}S(-q_i) \to S\to S/J \to 0.\end{equation}

\noindent The generator and relation degrees satisfy $q_i + p_i = 2d+3$, and we set $r_i = p_i - q_i$ and refer to the $r_i$ as \emph{diagonal degrees}. In this way, the Hilbert function is completely specified by the degrees of generators and relations $D = (P,Q)$ with $P = (p_1,p_2,\ldots, p_{2k+1})$ and $Q = (q_1, q_2, \ldots, q_{2k+1})$. If row $i$ of the partition has $a_i$ filled and $b_i$ unfilled blocks, then $r_{i +1} = b_i - a_i + 1$. Additionally, if $r_i + r_j = 0$, then we can delete the pair $(p_i,p_j)$ from $P$ and the pair $(q_i,q_j)$ from $Q$ without changing the Hilbert function. By iteratively removing pairs from $P$ and $Q$, we obtain a minimal set of generator and relation degrees $D_{\min} = (P_{\min},Q_{\min})$.
\end{remark}

\begin{example}\label{ex:ExamplePartition}
Consider the case with $d = 4$ and minimal generator degree $3$. The Hilbert functions of $S/J$ with socle degree $2d$ are in one-to-one correspondence with self-complementary partitions of $6 \times 4$ blocks. We construct the minimal generators of the ideal corresponding to the partition shown in Figure \ref{fig:ExamplePartition}.  

\begin{figure}
\begin{minipage}{0.48\textwidth}
\begin{tikzpicture}
\draw[very thick] (0,0) rectangle (4,6);
\fill[gray] (0,0) rectangle (1,6);
\fill[gray] (1,0) rectangle (2,4);
\fill[gray] (2,0) rectangle (3,2);
\foreach \x in {1,...,3}
    \draw[very thick] (\x,0)--(\x,6);
\foreach \y in {1,...,5}
    \draw[very thick] (0,\y)--(4,\y);
\end{tikzpicture}
\end{minipage}
 \begin{minipage}{0.48\textwidth}
    \begin{tabular}{c || c c c c c c c c}
        $i$ & 1 & 2 & 3 & 4 & 5 & 6 & 7 \\
         \hline
        $r_i$ & 5 & 3 & 3 & 1 & 1 & -1 & -1\\
        $q_i$ &3 & 4 & 4 & 5 & 5 & 6 & 6 \\ 
        $p_i$ & 8 & 7 & 7 & 6 & 6 & 5 & 5\\
    \end{tabular}
 \end{minipage}

\caption{Partition of a $6\times 4$ block corresponding to the Hilbert function $T = (1,3,6,9,10,9,6,3,1)$ of a Gorenstein ideal with socle degree $8$ and minimal generator degree 3. The table displays the degrees of generators ($q_i$), degrees of relations ($p_i$), diagonal degrees $(r_i)$. Since $r_4 + r_6 = r_5+r_7 = 0$, there is an ideal $J$ generated by a cubic and two quartics which gives rise to this Hilbert function. }\label{fig:ExamplePartition}
\end{figure}

This partition corresponds to the Hilbert function $T = (1,3,6,9,10,9,6,3,1)$. Since $r_4 + r_7 = 0$ and $r_5 + r_6 = 0$, we see that an ideal $J$ giving rise to the Hilbert function $T$ can be minimally generated by one cubic and two quartics. 
\end{example}

\begin{theorem}[{\cite[Theorem 2.7 and Theorem 3.8]{Diesel1996Irreducibility}}]\label{thm:GorT_irreducible}
    Let $T$ be the Hilbert function of a Gorenstein ideal $J\subseteq S$ and $Gor_T$ the scheme of all Gorenstein ideals with Hilbert function $T$. Let $D$ be a set of generator and relation degrees giving rise to the Hilbert function $T$ and $Gor_D$ the family of Gorenstein ideals with generator and relation degrees given by $D$. Then, $\overline{Gor_D}$ is irreducible and equal to the union $\bigcup_{D' \supseteq D}Gor_{D'}$, where $D'$ ranges over sets of generator and relation degrees giving rise to $T$ and containing $D$. In particular, $Gor_T$ is irreducible and $Gor_T = \overline{Gor_{D_{\min}}}$ where $D_{\min}$ is the minimal set of generator and relation degrees giving rise to the Hilbert function $T$
        
\end{theorem} 

We apply Theorem \ref{thm:GorT_irreducible} to show that generic points on the boundary of the image of $\Phi_d$ give rise to Gorenstein ideals with minimal sets of generator and relation degrees. 
We still use the notation 
\[ M := \partial \, \overline{\mathrm{int}(\im(\Phi_d))} \bigcap \mathrm{int}(\Sigma_{2d})\subseteq S_{2d} . \] 
Using Proposition \ref{prop:IntroGorenstein} we associate to a generic element $p = \sum_{i = 0}^{d}h_i^2 \in M$ the Gorenstein ideal $I = I(L_{\mathbf{h}})$, where $L_\mathbf{h} \in S_{2d}^*$ annihilates $\langle h_0,h_1,\ldots, h_{d}\rangle_{2d}$.   

\begin{proposition}\label{prop:PrunedSuffice}
    If $p \in M$ is generic, then there exist $h_0,h_1,\ldots, h_{d} \in S_{d}$ such that $p = \sum_{i = 0}^{d}h_i^2$ and $\langle h_0,h_1,\ldots, h_{d}\rangle$ is contained in a Gorenstein ideal $J$ for which the set $D$ of generator and relation degrees is minimal. 
\end{proposition}

\begin{proof}
    Since $M$ has pure codimension $1$ (see~\cite[Lemma 4.5.2]{BCR1998RealAlgebraicGeometry}), its Zariski closure is a hypersurface. This hypersurface is contained in the branch divisor of $\Phi_d$ so that for a generic $p\in M$ with $p = \Phi_d(h_0,\ldots,h_d)$ the ideal $\langle h_0,h_1,\ldots,h_d \rangle$ is contained in a Gorenstein ideal with socle in degree $2d$ because its graded part of degree $2d$ is the tangent hyperplane $T_p M$ to $M$ at $p$. So for any irreducible component of the Zariski closure of $M$ there is a Hilbert function $T$ of a Gorenstein ideal with socle in degree $2d$ with the property that $T$ is the Hilbert function of $I(L_\mathbf{h})$ for the linear function $L_\mathbf{h}$ defining $T_p M$ for a generic point $p\in M$ on that irreducible component (and hence any generic point on that component).

    Now pick a generic point $p\in M$. Then $p$ is contained in only one irreducible component of the Zariski closure of $M$ and the Gorenstein ideal $I(L_\mathbf{h})$ with $T_p M = \cV(L_\mathbf{h})$ is in $Gor_{D_{min}}$ for the minimal set of generator and relation degrees for the Hilbert function $T$ corresponding to the irreducible component containing $p$ because $Gor_{D_{min}}$ is an open and dense subset of $Gor_T$ by \Cref{thm:GorT_irreducible}. Since $\langle h_0',h_1',\ldots,h_d'\rangle_{2d} \subset T_p M$ for any tuple $\mathbf{h'} = (h_0',h_1',\ldots,h_d')$ with $\Phi_d(\mathbf{h'}) = p$, the claim that $\langle h_0',h_1',\ldots,h_d'\rangle \subset J$ for a Gorenstein ideal $I(L_\mathbf{h}) = J\in Gor_{D_{min}}$ follows.
\end{proof}

We are only interested in Gorenstein ideals $J$ for which $\langle J_d \rangle_{2d} \subseteq S_{2d}$ is a hyperplane. Since the partition gives an explicit resolution of the ideal, we can eliminate many Gorenstein ideals from consideration. As an illustration, we show below that $J$ cannot have $d+1$ generators of degree $d$. This corresponds to the partition of a $2d \times 2$ block where $a_1 = a_2 = \dots = a_d = 0$ and $a_{d+1} = a_{d+2} = \dots = a_{2d} = 2$. 

\begin{lemma}
Suppose that $d \geq 2$. If $J$ is a Gorenstein ideal with $d+1$ generators of degree $d$ and $d$ generators of degree $d+2$, then $J_d$ does not generate a hyperplane in $S_{2d}$. 
\end{lemma}

\begin{proof}
We compute dimensions. By Proposition \ref{prop:GorensteinPartitions}, there are $d+1$ generators of degree $d$ and $d$ relations of degree $d+1$. These relations cannot involve the generators of degree $d+2$. Moreover, by the free resolution \eqref{eq:Ternary_Free_Res} there are no relations among these relations. So, the $d+1$ generators of degree $d$ generate a space of dimension

\[ (d+1)\binom{d+2}{2} - d\binom{(d-1) + 2}{2} = \frac{1}{2}(d+1)(3d+2).\]
On the other hand, $\dim S_{2d} = (d+1)(2d+1)$. Since $\frac{1}{2}(d+1)(3d+2) < \dim S_{2d} - 1$ when $d \geq 2$, $J_d$ cannot generate a hyperplane in $S_{2d}$. 
\end{proof}

More generally, if $J$ is a Gorenstein ideals with generator degrees $Q = \{q_1,q_2,\ldots, q_s\}$ and relation degrees $P = \{p_1,p_2,\ldots, p_s\}$, set $m = \min \{q_i \; \vert \; q_i \geq (d+1)\}$. Then, 

\begin{equation}\label{eq:Initial_Bound}
\dim \langle J_d \rangle_{2d} \leq \sum_{i: q_i \leq d} \binom{2 + (2d- q_i)}{2} - \sum_{i : p_i < m}\binom{2 + (2d - p_i)}{2}.\end{equation}

On the other hand, enforcing that the $h_i$ are basepoint free yields a lower bound on the codimension (and therefore an upper bound on the dimension) of the degree $d$ part of a Gorenstein ideal containing them. 

\begin{theorem}[{\cite[Theorem 4.1]{blekherman2015positive}}]
Let $W$ be a subspace of $\bbC[x_1,x_2,x_3]_{d}$ with $d\geq 3$ such that $\cV_\bbC(W) = \emptyset$. If $J = \langle W \rangle$ has $\operatorname{codim} J_{2d} \geq 1$ then $\operatorname{codim} J_d \geq 3d - 2$. 
\end{theorem}

\section{Proof of Theorem \ref{thm:IntroMainThm}}\label{sec:MainThmProof}

In this section, we prove Theorem \ref{thm:IntroMainThm}. In the cases $2d = 10$ and $2d = 12$, our strategy is to bound the length of an element $p = \sum_{i = 0}^{d}h_i^2$ when the forms $h_i$ are contained in the degree $d$ part of a Gorenstein ideal $J$ with socle degree $2d$. To make this strategy precise, we define the Pythagoras number of a real projective subvariety $X \subseteq \bbP^n$ with homogeneous coordinate ring $R$ to be the minimum positive $r$ such that any $f \in R_2$ which is a sum of squares of elements of $R_1$ can be written as a sum of squares of at most $r$ elements of $R_1$. The Pythagoras number of $X$ is denoted $\py(X)$. 
Note that $\py(n,2d) = \py(\nu_d(\bbP^{n-1}))$, where $\nu_d$ is the $d^{th}$ Veronese map.  

The following results on Pythagoras numbers of varieties are established in \cite{Blekherman2022QuadraticPersistence}. Recall that a variety $X\subseteq \bbP^n$ is nondegenerate if it is not contained in a hyperplane and totally real if the real points $X(\bbR)$ form a Zariski dense subset of $X$.

\begin{theorem}[{\cite{Blekherman2022QuadraticPersistence}}]\label{thm:VarietyPyBounds}
Let $X \subseteq \bbP^n$ be a nondegenerate, irreducible, totally real variety. Then, 

\begin{enumerate}
\item $\py(X) = \dim(X) + 1$ if and only if $\deg(X) = 1 + \operatorname{codim}(X)$
\item If $X$ is arithmetically Cohen-Macaulay (aCM), then $\py(X) = 2 + \dim(X)$ if and only if $\deg(X) = 2 + \operatorname{codim}(X)$ or $X$ is a codimension-one subvariety of a variety of minimal degree.
\end{enumerate}
\end{theorem}

We will also use the following:

\begin{lemma}[{\cite[Lemma 2.5]{Blekherman2022QuadraticPersistence}}]\label{lem:py_nonincrease}
Let $X \subseteq X' \subseteq \bbP^n$ be nondegenerate, irreducible, totally real varieties. Then $\py(X) \leq \py(X')$.
\end{lemma}

Our strategy for the cases $2d = 10$ and $2d = 12$ is summarized as follows. Let $p = \sum_{i = 0}^{d}h_i^2$, $J$ be a Gorenstein ideal satisfying the dimension conditions established in Section \ref{sec:GorensteinDimensions}, and $\{t_1,t_2,\ldots, t_r\}\subseteq J_d$ be a basis of $J_d$. We show that in the relevant cases, the subalgebra $\bbR[t_1,t_2,\ldots, t_r]$, with grading $\deg(t_i) = 1$ for each $i$, is the homogeneous coordinate ring of a variety $X \subseteq \bbP^{r-1}$ with $\py(X)\leq d$. Because $h_i \in J_d$ for all $i = 0,\ldots, d$, it then follows that $p$ has length at most $d$, the desired contradiction.  

We will use the language of toric varieties to construct varieties $X'$ which bound the Pythagoras numbers. For a lattice polytope $K \subseteq \bbR^n$ we set $X_K$ to be the Zariski closure of the image of the map 

\[ (\bbC^*)^{n} \ni (t_1,t_2,\ldots, t_n) \longmapsto \left(\prod_{i =1}^nt_i^{\alpha_i}\right)_{\alpha \in K \cap \bbZ^n} \subseteq \bbP^{|K \cap \bbZ^n| - 1}.\]
This is the image of the projective toric variety $X_{
\cN_K}$ associated to the normal fan $\cN_K$ of $K$ in $\bbP^{|K \cap \bbZ^n| - 1}$. The variety $X_K$ is isomorphic to $X_{\cN_K}$ if $K$ is very ample. A general reference for the theory of toric varieties is \cite{CoxLittleSchenckToricVarieties}.

For a positive integer $m$ and let $T_m = \conv((0,0),(0,m),(m,0)) \subseteq \bbR^2$ be the triangle associated to the $m^{th}$ Veronese surface $\nu_m(\bbP^2)$. Proposition \ref{prop:Toric_Containment} below shows that in the relevant cases the variety obtained from $\bbR[t_1,t_2,\ldots, t_r]$ can be embedded in a toric variety corresponding to a polytope which consists appropriately stacked copies of $T_m$. Before stating and proving Proposition \ref{prop:Toric_Containment}, we provide an example of the construction.

\begin{example}\label{ex:Toric_Construction} Consider a Gorenstein ideal $J$ with socle degree $2d = 10$ arising as the complete intersection of two quartics and a quintic. Let $f_1,f_2$ be quartic generators and $f_3$ the quintic generator of $J$ so that a basis for $J_5$ is given by 

\[\begin{array}{c c c c}
t_{1,1,0} = x_1f_1 & t_{1,0,1} = x_2f_1 & t_{1,0,0} = x_3f_1 \\ t_{2,1,0} = x_1f_2 & t_{2,0,1} = x_2f_2 & t_{2,0,0} = x_3f_2 & t_{3,0,0} = f_3.
\end{array}\]

Let $R = \bbR[J_5]$ be the algebra generated by the degree $5$ elements of $J$. There are nontrivial quadratic relations among the $t_{i,j,k}$ when $i = 1,2$. For example,

\[t_{1,1,0}t_{2,0,1} - t_{1,0,1}t_{2,1,0} = x_1f_1x_2f_2 - x_2f_1x_1f_2 = 0.\]

To formally describe the quadratic relations among these basis elements, we first treat each $f_i$ as an independent variable $y_i$. Disregarding $y_3$ for the moment, we get a variety defined by a monomial embedding, and in fact this is the toric variety
$X_K$ with the polytope $K$ given by \[K = \conv(\{(0,0,0), (1,0,0), (0,1,0),(0,0,1),(1,0,1),(0,1,1)\})\subseteq \bbR^3,\]
where the parametrization of $X_K$ is given in Cox coordinates \cite[Chapter 5.4]{CoxLittleSchenckToricVarieties}. This is a variety of minimal degree, and of dimension $3$. If we add $y_3$, then we get a cone $Y$ over $X_K$, which is still a variety of minimal degree, of dimension $4$. Specializing, $y_i=f_i$ can only add additional elements to the ideal of $Y$, and therefore $\Proj(R)$ is contained in a variety of minimal degree of dimension $4$. Thus $\py (\Proj(R))\leq 5$.

\end{example}

\begin{proposition}\label{prop:Toric_Containment}
Fix $d$. Let $J\subseteq S$ be a Gorenstein ideal with socle degree $2d$ and generator and relation degrees $q_1\leq q_2 \leq \ldots \leq q_s$ and $p_1 \geq p_2 \geq \ldots \geq p_s$, respectively. Let $\{f_1,f_2,\ldots, f_{s}\} \subseteq S$ be a set of generators of $J$ with $\deg(f_i) = q_i$. Let $s'$ be such that $q_i < d$ for $1\leq i \leq s'$  and $q_j \geq d$ for $j > s'$. 
Set 

\[\Omega = \{(i,j,k) \in \bbZ^3 \; \vert \; 1\leq i \leq s' \text{ and } (j,k) \in T_{d-q_i} \cap \bbZ^2\}\] 
and denote $|\Omega| = r$. For each $\omega = (i,j,k) \in \Omega$, set

\[t_\omega = x_1^jx_2^kx_3^{d - q_i - j - k}f_i \in S_d.\]

Then, the map $\bbR[y_{\omega} \; \vert \; \omega \in \Omega] \to \bbR[t_\omega \;  \vert \; \omega \in \Omega]$ which sends $y_\omega \mapsto t_\omega$ embeds $\Proj(\bbR[t_{\omega}\; \vert \; \omega \in \Omega]) \subseteq X_K \subseteq \bbP^{r-1}$, where $X_K$ is the toric variety associated to the polytope

\[K = \conv\left( \bigcup_{i = 1}^{s'} T_{d-q_i} \times e_{i-1} \right)\]
where we set $e_0 = 0 \in \bbR^{s'-1}$ and where $e_1,e_2,\ldots, e_{s'-1}$ are the standard basis vectors of $\bbR^{s'-1}$. 
\end{proposition}

\begin{proof}[Proof of Proposition \ref{prop:Toric_Containment}]
We show that the quadratic relations on $X_K$ vanish on the generators $t_\omega$ of $R = \bbR[t_\omega \; \vert \; \omega \in \Omega ]$. For $(u_1,u_2, z_1,z_2,\ldots, z_{s'}) \in \bbC^{2+s'}$ set $v_{1,j,k} = u_1^ju_2^k$ and $v_{i,j,k} = z_{i-1}u_1^ju_2^k$. Then, $X_K$ is the Zariski closure of the map

\[\bbC^{2 + s'}\ni(u_1,u_2, z_1,z_2,\ldots, z_{s'}) \longmapsto \left(v_{\omega} \; \vert \;  \omega\in \Omega\right). \]

The quadratics in the ideal defining $X_K$ are generated by the binomials $v_{i_1,j_1,k_1}v_{i_2,j_2,k_2} - v_{i_3,j_3,k_3}v_{i_4,j_4,k_4}$, where $(j_1+j_2, k_1+k_2) = (j_3 + j_4,k_3 + k_4)$ and $\{i_1,i_2\} = \{i_3,i_4\}$.  For any such set of indices, we have 

\begin{multline*}
t_{i_1,j_1,k_1}t_{i_2,j_2,k_2} - t_{i_3,j_3,k_3}t_{i_4,j_4,k_4} = \\ x_1^{j_1 + j_2}x_2^{k_1+k_2}x_3^{2d - (q_{i_1} + q_{i_2} + j_1+j_2 + k_1 + k_2)} f_{i_1}f_{i_2} - x_1^{j_3 + j_4}x_2^{k_3+k_4}x_3^{2d - (q_{i_3} + q_{i_4} + j_3+j_4 + k_3 + k_4)} f_{i_3}f_{i_4} = 0    
\end{multline*}

\end{proof}

We note that the containment in Proposition \ref{prop:Toric_Containment} cannot be strengthened to an equality in general, as shown by the following example.

\begin{example}
Let $J = \langle x_1^4,x_2^4,x_3^5 \rangle \subseteq S$, and set

\[\begin{array}{c c c}
t_1 = x_1^5, & t_2 = x_1^4x_2 & t_3 = x_1^4x_3\\
t_4 = x_1x_2^4, & t_5 = x_2^5 & t_6 = x_2^4x_3
\end{array}\]
There is the degree 4 relation $t_2^4 = t_1^3t_4$, which is not an equation in the defining ideal of $X_K$, constructed as in Example \ref{ex:Toric_Construction}. 
\end{example}

Our application of Proposition \ref{prop:Toric_Containment} will be to bound the length of elements $p = \sum_{i = 0}^{d}h_i^2$, where $h_i \in J_d$ for some Gorenstein ideal $J$. In the cases of interest, the toric variety $X_K$ constructed in Proposition \ref{prop:Toric_Containment} will have easily computable (co)dimension and degree. 

We are now prepared to prove Theorem \ref{thm:IntroMainThm}.

\begin{proof}[Proof of Theorem \ref{thm:IntroMainThm}]
Let $p = \sum_{i = 0}^{d}h_i^2$ be a general element of $\partial \, \overline{ \int \im \Phi_d} \cap \int(\Sigma_{2d})$. Let $q \in \int(\Sigma_{2d})$ be as in Proposition \ref{prop:IntroArgmin} and $L \in S_{2d}^*$ be the linear functional with $L(g) = \langle q-p,g\rangle$ for all $g \in S_{2d}$. By Proposition \ref{prop:IntroGorenstein} and Part (1) of Proposition \ref{prop:Nonneg_SOS_Syzygies}, there is a Gorenstein ideal $J \subseteq S$ such that $J = I(L)$. By Proposition \ref{prop:PrunedSuffice}, it suffices to consider the Gorenstein ideals $J$ with minimal generator and relation degrees. 

A straightforward python script\footnote{available at \url{https://github.com/Alex-Dunbar/pythagoras-numbers-ternary-forms.git}} enumerates all possible Hilbert functions $T_J$ for Gorenstein ideals $J \subseteq S$, computes 
the minimal generator and relation degrees of $J$, and removes ideals $J$ which do not satisfy the basepoint freeness and dimension conditions derived in Section \ref{sec:GorensteinDimensions}:

\begin{itemize}
\item $J$ has at least three generators of degree at most $d$,
\item if $J$ has exactly three generators $f_1,f_2,f_3$ of degree at most $d$ with no common (complex) zero, then $\sum_{i = 1}^3\deg(f_i) = 2d + 3$,  
\item $\mathrm{dim} J_d \geq d+1$,
\item the bound \eqref{eq:Initial_Bound} is larger than $\dim S_{2d} - 1$,
\item and $\mathrm{codim} J_d \geq 3d-2$.
\end{itemize}

\subsection{Degree \texorpdfstring{$2d = 10$}{2d = 10}}

There are three possible sets of generator and relation degrees for Gorenstein Ideals:

\begin{enumerate}
\item $Q_1 = \{3,5,5\}, P_1 = \{10,8,8\}$
\item $Q_2 = \{4,4,5\}, P_2 = \{9,9,8\}$
\item $Q_3 = \{4,5,5,5,7\}, P_3 = \{9,8,8,8,6\}$
\end{enumerate}

In each case, we show that the length of a sum of squares of elements of $\bbR[J_d]_d$ is at most $d$, contradicting part (4) of Proposition \ref{prop:Nonneg_SOS_Syzygies}.

\subsubsection{Case (1):} Let $f$ be the cubic generator and $q_1,q_2$ the quintic generators so that a basis for $J_5 \subseteq S_5$ is given by

\[\begin{array}{c c c c}
t_1 = x_1^2f & t_2 = x_1x_2f & t_3 = x_1x_3f & t_4 = x_2^2f\\
t_5 = x_2x_3f & t_6 = x_3^2f & t_7 = q_1 & t_8 = q_2
\end{array}\]
We apply Proposition \ref{prop:Toric_Containment} to the subalgebra $R = \bbR[t_1,t_2,\ldots, t_6] \subseteq S$ to obtain that $\Proj(R) \subseteq X_K$, where $X_K = \nu_2(\bbP^2)$ and so $X_K$ is a variety of minimal degree. Adding $t_7$ and $t_8$ we see that $\Proj \bbR[J_5]$ is contained in a double cone $Y$ over $X_K$ (i.e. $Y$ is a cone over a cone over $X_K$). Then $Y$ is a variety of minimal degree of dimension $4$. Thus $\py (\Proj(\bbR[J_5]))\leq 5$.

\subsubsection{Case (2):}

Let $f_1,f_2$ be quartic generators and $q$ the quintic generator so that a basis for $J_5$ is given by 

\[\begin{array}{c c c c}
t_1 = x_1f_1 & t_2 = x_2f_1 & t_3 = x_3f_1 \\ t_4 = x_1f_2 & t_5 = x_2f_2 & t_6 = x_3f_2 & t_7 = q.
\end{array}\]

We established in Example \ref{ex:Toric_Construction} that $\py (\Proj(\bbR[J_5]))\leq 5$. 

\subsubsection{Case (3):} Let $f$ be the quartic generator and $g_1,g_2,g_3$ quintic generators so that 

\[h_0 = \ell_0f,\, h_1 = \ell_1f,\, h_2 = \ell_2f,\, h_3 = g_1,\, h_4 = g_2,\,  h_5 = g_3,\]
where $\{\ell_0,\ell_1,\ell_2\}$ is a basis of $S_1$. Since there is a relation in degree 6, there are linear forms $m_0,m_1,\ldots, m_5$ such that

\begin{equation}\label{eq:2d=10_degree_6_syzygy}\ell_0m_0f + \ell_1 m_1f + \ell_2m_2f +  m_3g_1 + m_4g_2 + m_5g_3 = 0.\end{equation}
Since $\{\ell_0,\ell_1,\ell_2\}$ is a basis of $S_{1}$, it follows that there are $\alpha_{i,j} \in \bbR$ such that $m_{i} = \sum_{j = 0}^2 \alpha_{i,j}\ell_j$. Rearranging \eqref{eq:2d=10_degree_6_syzygy} and multiplying by $f$ then gives polynomials $a_0,a_1,a_2 \in \langle f,g_1,g_2,g_3 \rangle_5$ such that

\[f\ell_0 a_0 + f\ell_1a_1 + f\ell_2a_2 = h_0a_0 + h_1a_1 + h_2a_2 =0.\]

So, by Proposition \ref{lem:syzygy_in_ideal}, $\sum_{i = 0}^5 h_i^2$ can be written as a sum of at most five squares.

\subsection{Degree \texorpdfstring{$2d = 12$}{2d = 12}} The possible Gorenstein ideals have generator and relation degrees given by

\begin{enumerate}
\item {$Q_1 = \{3,6,6\}$, $P_1 = \{12, 9,9\}$}  
\item $Q_2 = \{4,4,6,6,10\}$, $P_2 = \{11,11,9,9,5\}$
\item $Q_3 = \{4,5,6\}$, $P_3 = \{11,10,9\}$
\item { $Q_4 = \{4,6,6,6,8\}$, $P_4 = \{11,9,9,7\}$}
\item  $Q_5 = \{5,5,5\}$, $P_5 = \{10,10,10\}$ 
\item $Q_6 = \{5,5,6,6,8\}$, $P_6 = \{10,10,9,9,7\}$
\item $Q_{7} = \{5,6,6,6,6,8,8\}$, $P_{7} = \{10,9,9,9,9,7,7\}$  
\end{enumerate}

\subsubsection{Case (1):} 

Let $f$ be the cubic generator and $q_1,q_2$ two sextic generators. Then, a basis for $J_6$ is given by 

\[\begin{array}{c c c c c c}
t_1 = x_1^3f& t_2 = x_1^2x_2f & t_3 = x_1^2x_3f & t_4 = x_1x_2x_3f & t_5 = x_1x_2^2f & t_6 =x_1x_3^2f \\ t_7 = x_2^3f & t_8 = x_2^2x_3f & t_9 = x_2x_3^2f & t_{10} = x_3^3f & t_{11} = q_1 & t_{12} = q_2.
\end{array}
\]
We apply Proposition \ref{prop:Toric_Containment} to the subalgebra $R = \bbR[t_1,t_2,\ldots, t_{10}]\subseteq S$ to obtain that $\Proj(R) \subseteq \nu_3(\bbP^2)$, which is an aCM variety of almost minimal degree. By adjoining $t_{11}$ and $t_{12}$ to the coordinate ring, we see that $\Proj(\bbR[J_6])$ is contained in the variety $Y$ obtained as a double cone over $\nu_3(\bbP^2)$. Therefore $Y$ is an aCM variety of almost minimal degree and dimension $4$. So, $\py(\Proj(\bbR[J_6])) \leq 6$. 

\subsubsection{Case (2):} Since there is a relation in degree $5$, there is a cubic $f$ which divides both quartic generators. In particular, by Case (1) the length of $p$ is at most $6$. 

\subsubsection{Case (3):} Let $f$ be the quartic generator, $g$ the quintic generator, and $q$ the sextic generator. A basis of $J_6$ is given by 

\[
\begin{array}{c c c c c}
t_1 = x_1^2f & t_2 = x_1x_2f & t_3 = x_1x_3f & t_4 = x_2^2f & t_5 = x_2x_3f\\
t_6 = x_3^2f & t_7 = x_1g & t_8 = x_2g & t_9 = x_3g & t_{10} = q.
\end{array}
\]
We apply Proposition \ref{prop:Toric_Containment} to the subalgebra $R = \bbR[t_0,t_1,\ldots, t_9]\subseteq S$ to obtain that  $\Proj(R)\subseteq X_K \subseteq \bbP^8$, where $X_K$ is  determined by the polytope

\[K = \conv\left(\{(0,0,0),(2,0,0),(0,2,0),(0,0,1),(1,0,1),(0,1,1)\}\right)\]
Using Macaulay2 \cite{M2,DepthSource}, we see that $X_K$ is arithmetically Cohen-Macaulay, and of almost minimal degree. By adjoining $t_{10}$ to the coordinate ring, we see that $\Proj(\bbR[J_6])$ is contained in the variety $Y$ obtained by taking a cone over $X_K$. Therefore $Y$ is an aCM variety of almost minimal degree of dimension $4$ and $\py(\Proj(\bbR[J_6]))  \leq 6$.

\subsubsection{Case (4):} Let $f$ be the quartic generator and $q_1,q_2,q_3$ the sextic generators. A basis for $J_6$ is given by 

\[
\begin{array}{c c c c c}
t_1 = x_1^2f & t_2 = x_1x_2f & t_3 = x_1x_3f & t_4 = x_2^2f\\
t_5 = x_2x_3f & t_6 = x_3^2f & t_7 = q_1 & t_8 = q_2 & t_9 = q_3
\end{array}
\]

By Proposition \ref{prop:Toric_Containment}, if $R = \bbR[t_1,t_2,\ldots, t_6]\subseteq S$, then $\Proj(R) \subseteq \nu_2(\bbP^2)$. By adjoining the variables $t_7,t_8,$ and $t_9$, we obtain that $\Proj(\bbR[J_6])$ is contained in the variety $Y$ obtained by taking a triple cone over $\nu_2(\bbP^2)$. Then, $Y$ is a variety of minimal degree and dimension $5$ so that $\py(\Proj(\bbR[J_6]))\leq 6$

\subsubsection{Case (5):} Let $q_1,q_2,q_3$ be the three quintic generators. A basis for $J_6$ is given by 

\[
\begin{array}{c c c}
t_1 = x_1q_1 & t_2 = x_2q_1 & t_3 = x_3q_1\\
t_4 = x_1q_2 & t_5 = x_2q_2 & t_6 = x_3q_2\\
t_7 = x_1q_3 & t_8 = x_2q_3 & t_9 = x_3q_3
\end{array}
\]

Applying Proposition \ref{prop:Toric_Containment} to the subalgebra $R = \bbR[t_1,t_2,\ldots t_9]\subseteq S$ gives that $\Proj(R) \subseteq X_K$, where $X_K$ is the toric variety associated to the polytope

\[K = \conv((0,0,0,0),(1,0,0,0),(0,1,0,0),(0,0,1,0),(1,0,1,0),(0,1,1,0),(0,0,0,1),(1,0,0,1),(0,1,0,1)).\]

Using Macaulay2, we verify that $X_K$ is aCM and of almost minimal degree so that $\py(X_K) = \dim(X_K) + 2 = 4+ 2 = 6$. Therefore $p$ has length at most $6$. 

\subsubsection{Case (6):} Let $f_1,f_2$ be the quintic generators and $q_1,q_2$ the sextic generators. A basis of $J_6$ is given by 

\[
\begin{array}{c c c c}
t_1 = x_1f_1 & t_2 = x_2f_1 & t_3 = x_3f_1 & t_4 = x_1f_2\\ t_5 = x_2f_2 & t_6 = x_3f_2 & t_7 = q_1 & t_8 = q_2.
\end{array}\]

Applying Proposition \ref{prop:Toric_Containment} to the subalgebra $R = \bbR[t_1,t_2,\ldots, t_6] \subseteq S$ yields that $\Proj(R) \subseteq X_K$, where $X_K$ is the toric variety corresponding to the polytope 

\[K = \conv(\{(0,0,0), (1,0,0), (0,1,0),(0,0,1),(1,0,1),(0,1,1)\})\subseteq \bbR^3.\]
By adjoining the variables $t_7$ and $t_8$ to the coordinate rings, we obtain that $\Proj(\bbR[J_6]) \subseteq Y$, where $Y$ is obtained as a double cone over $X_K$. Then $Y$ is a variety of minimal degree and dimension $5$ so that $\py(\Proj(\mathbb{R}[J_6])) \leq 6$.

\subsubsection{Case (7):} This case is similar to $2d = 10$, Case (3). Let $f$ be the quintic generator and $g_1,g_2,g_3,g_4$ sextic generators such that 

\[\begin{array}{c c c c }
h_0 = \ell_0 f & h_1 = \ell_1 f & h_2 = \ell_2 f\\
h_3 = g_1 & h_4 = g_2 & h_5 = g_3 & h_6 = g_4
\end{array}\]
for some basis $\{\ell_0,\ell_1,\ell_2\}$ of $S_1$. Since there is a relation in degree $7$, there are $\alpha_{i,j} \in \bbR$, for $0\leq i \leq 6$ and $0\leq j \leq 2$ such that 

\begin{equation}\label{eq:deg12}\sum_{i = 0}^2\left(\sum_{j = 0}^2\alpha_{i,j}\ell_j\ell_if \right) + \sum_{j= 0}^2\alpha_{3,j}\ell_jg_1 + \sum_{j= 0}^2\alpha_{4,j}\ell_jg_2 + \sum_{j= 0}^2\alpha_{5,j}\ell_jg_3 +\sum_{j= 0}^2\alpha_{6,j}\ell_jg_4 = 0.\end{equation}

Rearranging \eqref{eq:deg12} gives $a_0,a_1,a_2 \in \langle h_0,h_1,\ldots, h_6\rangle_6$ such that 

\[\ell_0a_0 + \ell_1a_1 + \ell_2a_2 = 0.\]

By multiplying by $f$, this implies that $h_0a_0 +h_1a_1 + h_2a_2 = 0$. Therefore, by Proposition \ref{lem:syzygy_in_ideal}, $p$ has length at most $6$. 

\subsection{Degree \texorpdfstring{$2d = 8$}{2d = 8}}

The only possible Gorenstein ideal $J$ has generator and relation degrees 

\[Q = \{3,4,4\}, \quad P = \{8,7,7\}\]
which is a complete intersection of a cubic and two quartics. Set $f$ to be the cubic generator of $J$. Then there is a basis $\{\ell_0,\ell_1,\ell_2\} \subseteq S_1$ of linear forms and $g,r \in S_{4}$ such that $J = \langle f,g,r\rangle$ and  

\[h_0 = f\ell_0 ,\enskip h_1 = f \ell_1 ,\enskip h_2 = f\ell_2,\enskip h_3 = g,\enskip \text{and } h_4 = r.\]
If we apply the approach of the previous subsections, the toric variety that we get is a double cone over $\nu_1(\P^2)$, which is $\P^4$. The Pythagoras number is $5$ in this case, and we do not get a contradiction.

Note that $\spann_{\bbR}\{h_0,h_1,\ldots, h_4\} = J_4$ and that these both coincide with $I(L)_4$, where $L \in S_8^*$ is the linear functional annihilating the hyperplane $\langle h_0,h_1,\ldots, h_{4}\rangle_8$. 

If there is a nontrivial quadratic $\sum_{0 \leq i,j \leq 4} c_{i,j}h_ih_j = 0$, then $p = \sum_{i = 0}^4h_i^2$ has length at most $4$ by Proposition \ref{lem:syzygy_in_ideal}. 

Otherwise, the only syzygies among $h_0,h_1,\ldots, h_4$ in the subalgebra $\bbR[h_0,h_1,\ldots, h_4] \subseteq S$ are those generated by those of the form $h_ih_j - h_jh_i$ for $0\leq i,j \leq 4$. We now claim that if $p =\sum_{i = 0}^{4} h_i^2$ is a second order stationary point of \eqref{eq:distance_to_image}, then for any $\hat{h}_0,\hat{h}_1,\ldots, \hat{h}_4$ with $\langle \hat{h}_0,\hat{h}_1,\ldots,\hat{h}_4\rangle = \langle h_0,h_1,\ldots, h_4\rangle$ and any $a_0,a_1,\ldots, a_4 \in S_4$ with $\sum_{i = 0}^{4}\hat{h}_ia_i = 0$, we have that $L(\sum_{i = 0}^{4}a_i^2)\leq 0$. Indeed, it follows from part (2) of Proposition \ref{prop:Nonneg_SOS_Syzygies} that the quadratic form on $S_4^5$ given by $(b_0,b_1,\ldots, b_4) \mapsto L(\sum_{i = 0}^4 b_i^2)$ is zero if and only if $b_0,b_1,\ldots, b_4 \in I(L)_4 =  \langle h_0,h_1,\ldots, h_4\rangle = \langle \hat{h}_0,\hat{h}_1,\ldots, \hat{h}_4\rangle$. In particular, the kernel of this quadratic form does not depend on the choice of basis $h_i$. Now, for fixed $a_0,a_1,\ldots, a_4 \in S_4$ with $\sum_{i = 0}^4h_ia_i = 0$, and for any nonzero constants $\gamma_0,\gamma_1,\ldots, \gamma_4$, we have that $\sum_{i = 0}^4 (\gamma_ih_i)(\frac{1}{\gamma_i}a_i) = 0$, and it follows that $L(a_i^2) \leq 0$ for each $i$ by taking $\gamma_j \to \infty$ for $j \not = i$. 

To conclude, it suffices to show that if $a \in S_4$ is any quartic, then there are $a_0,a_1,\ldots ,a_4 \in S_4$ such that $\sum_{i=0}^4a_ih_i = 0$ and $a = a_i$ for some $i$ , as this will imply that $L(a^2) \leq 0$ for all $a \in S_4$ and $-L \in \Sigma_8^*$, a contradiction with the construction of $L$. 

Let $a \in S_4$ have a real zero $v$, and let $m_0 \in S_1$ vanish on $v$. Then, because $a$ has a zero at $v$, the image of $a$ in $S/\langle m_0 \rangle$ factors as $m_1c_1$ mod $\langle m_0 \rangle$ for some $m_1 \in S_1$ and $c_1 \in S_3$. So, $a  = m_0c_0 + m_1c_1$, where $c_0 \in S_3$. Choose $m_2$ such that $\{m_0,m_1,m_2\}$ is a basis of $S_1$, and consider the change of basis of $I(L)_4$ given by $\ell_i \mapsto m_i$. Now

\[(-c_0m_2)(m_0h_0) + (-c_1m_2)(m_1h_1) + (a)(m_2h_2) = 0,\]
which implies that $L(a^2)\leq 0$. Now, every extreme ray of $\Sigma_8$ has a real zero, so it follows that $L(a^2) \leq 0$ for all $a \in S_4$ such that $a^2$ generates an extreme ray of $\Sigma_8$. It follows that $-L \in \Sigma_8^*$. \end{proof}

\printbibliography    

\end{document}